\documentclass[11pt]{amsart}
\usepackage{amsmath}
\usepackage{graphicx}
\usepackage{pst-all}
\usepackage{multido}
\definecolor{gris}{rgb}{.9, .9, .9}
\definecolor{bla}{rgb}{.5, .5, .5}
\newpsobject{mypmm}{psgrid}{subgriddiv=5,gridlabels=0pt,subgridcolor=gris,gridcolor=bla}
\newtheorem{thm}{Theorem}[section]
\newtheorem{cor}[thm]{Corollary}
\newtheorem{lem}[thm]{Lemma}
\newtheorem{prop}[thm]{Proposition}
\theoremstyle{definition}
\newtheorem{defn}[thm]{Definition}
\theoremstyle{remark}

\numberwithin{equation}{section}

\begin{document}

\title[A sharp smoothness of the conjugation]
{A sharp smoothness of the conjugation of class P-homeomorphisms to diffeomorphisms}
\author{Abdelhamid Adouani and Habib Marzougui}

\address{University of Carthage,
Faculty of Science of Bizerte, Department of Mathematics, Jarzouna. 7021. Tunisia.} \email{arbi.abdelhamid@gmail.com}
\email{hmarzoug@ictp.it; habib.marzougui@fsb.rnu.tn}


\begin{abstract} Let $f$ be a class $P$-homeomorphism of the circle. We prove that there exists a piecewise analytic homeomorphism that conjugate $f$ to a one-class $P$ 
with prescribed break points lying on pairwise distinct orbits. As a consequence, we 
give a sharp estimate for the smoothness of a conjugation of class $P$-homeomorphism $f$ of the circle 
satisfying the (D)-property (i.e. the product of $f$-jumps in the break points contained in a same orbit is trivial), 
to diffeomorphism. When $f$ does not satisfy the (D)-property the conjugating homeomorphism is never piecewise $C^{1}$
and even more it is not absolutely continuous function if the total product of $f$-jumps in all the break points is non-trivial. 
 
\end{abstract}
\maketitle

\section{Introduction}
\bigskip

Denote by  $S^{1} =
\mathbb{R}/\mathbb{Z}$  the circle and  $p :
\mathbb{R}\longrightarrow S^{1}$ the canonical projection. Let  $f$
 be an orientation preserving homeomorphism of  $S^{1}$.  The
homeomorphism  $f$  admits a lift  $\widetilde{f} :
\mathbb{R}\longrightarrow \mathbb{R}$  that is an increasing
homeomorphism of $\mathbb{R}$  such that $p \circ \widetilde{f} = f\circ p$.
Conversely, the projection of such a homeomorphism of
 $\mathbb{R}$ is an orientation preserving homeomorphism of
 $S^{1}$. The rotation number of a homeomorphism  $f$  of  $S^{1}$ is defined
as $\rho (f) = \underset{n\to +\infty}{\lim}\frac{\widetilde{f}^{n}(x) - x}{n}~(\textrm{mod } 1), \ x\in \mathbb{R}.$ This limit exists and is independent of the choice of the point $x$
 and the lift  $\widetilde{f}$  of  $f$. For example, if 
$R_{\alpha}: x\mapsto x+ \alpha~(\textrm{mod } 1)$ is the rotation by angle $\alpha$ then it is obviously that
 $\rho(R_{\alpha}) = \alpha~(\textrm{mod } 1)$. From the definition,  $\rho(h\circ f\circ h^{-1})= \rho(f)$ holds for any orientation preserving homeomorphism $h$
of  $S^{1}$. Assuming  $f$  is a  $C^{r}$-diffeomorphism ($r\geq 2$) and  $\rho(f)$  is irrational, Denjoy (\cite{aD32}) proved 
that: every  $C^{r}$-diffeomorphism  $f$ ($r\geq 2$) of $S^{1}$ with
irrational rotation number  $\rho (f)$ is topologically conjugate to the rotation  $R_{\rho (f)}$. This means that there exists an orientation preserving homeomorphism $h$ of $S^{1}$ such that  $f = h^{-1}\circ R_{\rho(f)}\circ h$. Denjoy noted that this result
can be extended (with the same proof) to a large class of circle homeomorphisms: \emph{the class  $P$} (see \cite{H},  Chapter VI) and in particular for
piecewise linear (PL) circle homeomorphisms.
\bigskip

\begin{defn}\label{d:10}  An orientation preserving homeomorphism $f$ of $S^{1}$ is called a \emph{class  $P$}-homeomorphism if it is derivable
except at finitely many points, the so called \textit{break points} of  $f$, at which left and right derivatives (denoted, respectively, by 
 $\mathrm{Df}_{-}$ and $\mathrm{Df}_{+}$) exist and such that
the derivative  $\mathrm{Df}: S^{1} \longrightarrow \mathbb{R}_{+}^{*}$  has the following properties:
\begin{itemize}
\item [-] There exist two constants  $0<a <b<+\infty$  such that: $a<\mathrm{Df}(x)<b,$  for every  $x$  where  $\mathrm{Df}$  exists,

\item [-]  $a <\mathrm{Df}_{+}(c)< b$ and  $a<\mathrm{Df}_{-}(c)< b$  at the break points  $c$.

\item [-]  $\log \mathrm{Df}$  has bounded variation on  $S^{1}$ (i.e.
the total variation of $\log \mathrm{Df}$ is finite).
\end{itemize}
\end{defn}

We pointed out that the third condition implies the two ones. Also notice that if $f$ is a class 
$P$-homeomorphism of $S^{1}$ which is $C^{1}$ on $S^{1}$ then $f$ is a $C^{1}$-diffeomorphism of $S^{1}$. 
\bigskip

\begin{defn}\label{d:11} An orientation preserving homeomorphism $f$ of $S^{1}$ is called piecewise linear (\emph{PL-homeomorphism}) if  $f$
is derivable except at finitely many break points  $(c_{i})_{0\leq i\leq p}$ of $S^{1}$ such that the derivative $\mathrm{Df}$
is constant on each  $]c_{i}, \ c_{i+1}[$.
\end{defn}
\medskip

Among the simplest examples of class $P$-homeomorphisms, we mention:
\begin{itemize}
 \item [\textbullet] $C^{2}$-diffeomorphisms,

\item [\textbullet] Piecewise linear $\textrm{PL}$-homeomorphisms, these are not C$^{2}$-diffeomorphisms.
\end{itemize}
\medskip

Denote by
\begin{itemize}
  \item [-] Homeo$_{+}(S^{1})$ the group of orientation-preserving homeomorphisms of $S^{1}$.

  \item [-] $\mathcal{P}(S^{1})$ the set of class $P$-homeomorphisms of $S^{1}$, it is a subgroup of Homeo$_{+}(S^{1})$.

 \item [-] $\textrm{PL}(S^{1})$ the set of  $\textrm{PL}$-homeomorphisms of $S^{1}$, it is a subgroup of
 $\mathcal{P}(S^{1})$  which contains rotations.
\end{itemize}
\medskip
\medskip

In this paper, we are mainly concerned with the sharp estimate for the smoothness of a conjugation of class $P$-homeomorphism 
with the ($D$)-property (see Definition \ref{d:14} and Theorem \ref{t:23}) to diffeomorphism. For class $P$-homeomorphism 
without the ($D$)-property, the conjugation is never piecewise $C^{1}$ (see Proposition \ref{p:18}) 
and even more, can be a singular function (see Corollary \ref{c:19}).
Before stating the main result, we need the following notations and definitions.

\
\\
 For $f\in \mathcal{P}(S^{1})$ and $x\in S^{1}$, denote by  
\begin{itemize}

 \item [-] $O_{f}(x): = \{f^{n}(x): \ n\in \mathbb{Z} \}$ called the \emph{orbit} of $x$ by $f$.

\item [-] $\sigma_{f}(x): \ = \ \frac{\textrm{Df}_{-}(x)}{\textrm{Df}_{+}(x)}$  called the \emph{ $f$-jump} in $x$.\\
 
\item [-] $\pi_{s, O_{f}(c)}(f) = \underset{x\in C(f)\cap O_{f}(c)}{\prod}\sigma_{f}(x)$,
for every $c\in C(f)$.

\item [-] $C(f) = \{c_{0}, c_{1}, c_{2},\dots, c_{p}\}$  the set of break points of $f$ in
$S^{1}$. 

\item [-] $c_{p+1}: = c_{0}$.
 
\item [\textbullet] $\pi_{s}(f)$ the product of $f$-jumps at the break points of $f$:
$$\pi_{s}(f) =  \underset{c\in C(f)}{\prod}\sigma_{f}(c).$$
\end{itemize}
\medskip

\begin{defn}\label{d:14} (\cite{AM2}) Let $f\in \mathcal{P}(S^{1})$. We say that $f$ has the \textit{($D$)-property} 
if the product of $f$-jumps in the break points on each orbit is trivial; that is $\pi_{s, O_{f}(c)}(f) = 1$, for every $c\in C(f)$.
\end{defn}
\medskip

In particular, if $f$ has the ($D$)-property, then $\pi_{s}(f) = 1$. Conversely, if all break points belong to the same orbit
and $\pi_{s}(f) = 1$ then $f$ has the ($D$)-property.
We established in (\cite{AM2}, Proposition 2.5) that $f$ has the ($D$)-property if and only if the number of break points of 
$f^{n}$ is bounded by some constant that doesn't depend on 
$n$. 
\medskip
 
\begin{defn}\label{d:13}$($Maximal connections$)$. Let $f\in \mathcal{P}(S^{1})$ and $c\in C(f)$. A \emph{maximal $f$-connection} of  $c$
 is a segment $$[f^{-p}(c),\dots, f^{q}(c)]:=  \{f^{s}(d): \ -p \leq s\leq q\}$$ of the orbit $O_{f}(c)$ which
contains all the break points of $f$ contained on $O_{f}(c)$ and such that $f^{-p}(c)$ (resp. $f^{q}(c)$) is the first (resp. last)
break point of  $f$ on  $O_{f}(c)$.
\end{defn}
We have the following properties:
\
\\
- Two break points of $f$ are on the same maximal $f$-connection, if and only if, they are on the same orbit.
\
\\
- Two distinct maximal $f$-connections are disjoint.
\bigskip

\textbf{Notations}. Let  $f\in \mathcal{P}(S^{1})$. We let
\medskip

\begin{itemize}

 \item [-]  $ M_{i}(f) = [c_{i},\dots, f^{N_{i}}(c_{i})], \ N_{i}\in \mathbb{N}^{*},$ the maximal $f$-connections of
$c_{i}\in C(f)$, ($0\leq i\leq p$).\\
\item [-]  $M(f) = \coprod_{i=0}^{p}M_{i}(f)$.

\
\\
So, we have the decomposition: $C(f) = \coprod_{i=0}^{p} C_{i}(f)$, where 

$C_{i}(f) = C(f)\cap M_{i}(f), \ 0\leq i\leq p$. In particular, $C_{i}(f)\subset  M_{i}(f)$.\\

\item [-] $N:= \underset{0\leq i\leq p}\max N_{i}$.
\end{itemize}
\
\\
Note that if $f$ has the ($D$)-property then:
\
\\
\begin{itemize}
\item [-] $\underset{d\in C_{i}(f)}\prod \sigma_{f}(d) = \underset{d\in
M_{i}(f)}\prod \sigma_{f}(d) =1$, for every $i=1,\dots, p$.
\end{itemize}

\
\\
%
%

Define \\
 \ \ $\bullet$ \ $\pi_{O_{f}(c_{i})}(f): =
\prod_{k=1}^{N_{i}}\sigma_{f^{N+1}}(f^{k}(c_{i})).$
By (\cite{AM2}, Lemma 2.7), we also have: $\pi_{O_{f}(c_{i})}(f): =
\underset{j\in \mathbb{Z}}\prod \big(\sigma_{f}(f^{j}(c_{i}))\big)^{j}.$

 $\bullet$ \ $\pi(f): = \prod_{i=0}^{p}\pi_{O_{f}(c_{i})}(f).$
\bigskip

Let $\sigma\in \mathbb{R}_{+}^{*}\setminus\{1\}$. 
We shall introduce the two following basic class $P$-homeomorphisms. 
Denote by
\bigskip

$\bullet$ \ $g_{\sigma}$ the orientation preserving homeomorphism of $S^{1}$ with lift ~ $\widetilde{g_{\sigma}}: \mathbb{R}\longrightarrow \mathbb{R}$ ~ restricted to $[0,1[$ is given by: \begin{equation*}\widetilde{g}_{\sigma}(x) =
\left(\frac{1-\sigma}{1+\sigma}\right)\left(x^{2} +\frac{2\sigma}{1-\sigma}x\right), \
x\in [0, 1[. \end{equation*}
\medskip

We identify $g_{\sigma}$ with its lift ~ $\widetilde{g_{\sigma}}$. Since  $g_{\sigma}(0) = 0$, $g_{\sigma}(1) = 1$ and $\sigma\neq 1$,  $g_{\sigma}\in \mathcal{P}(S^{1})$ with one break point $0$ and such that $\sigma_{g_{\sigma}}(0) = \sigma$. Moreover, $g_{\sigma}$ is quadratic on $S^{1}\backslash \{0\}$.
\medskip

$\bullet$ \ $h_{\sigma}$ the homeomorphism of $S^{1}$ with lift $\widetilde{h_{\sigma}}: \mathbb{R}\longrightarrow \mathbb{R}$ restricted to $[0,1[$ is given by:
$$\widetilde{h_{\sigma}}(x) = \frac{\sigma^{x}-1}{\sigma-1}, \ x\in [0, 1[.$$ We identify $h_{\sigma}$ with its lift $\widetilde{h_{\sigma}}$. Then  $h_{\sigma}\in \mathcal{P}(S^{1})$ with one break point $0$ and such that  $\sigma_{h_{\sigma}}(0) = \sigma$. Moreover, $h_{\sigma}$ is analytic on $S^{1} \backslash  \{0\}$.

\bigskip

\begin{defn} A homeomorphism $h$ of $S^{1}$ is called a $\mathrm{PQ}$-homeomorphism (resp. $\mathrm{PE}$-homeomorphism) of $S^{1}$ if 
$h = L\circ u$, where $L\in \mathrm{PL}(S^{1})$ 
and $u = R_c\circ g_{\sigma}\circ R_c^{-1}$ (resp. $R_c\circ h_{\sigma}\circ R_c^{-1}$), for some  
$\sigma\in \mathbb{R}_{+}^{*}\setminus\{1\}$ and $c\in S^{1}$.
\end{defn}  
\
\\
We are in the position to give our main result.
\bigskip

\begin{thm}\label{t:23} Let $f\in \mathcal{P}(S^{1})$ with the ($D$)-property and irrational 
rotation number. Then:
\begin{enumerate}
 \item [(i)]  If $\pi(f)\neq 1$, $f$ is conjugate to a diffeomorphism through a 
 $\mathrm{PQ}$(resp. $\mathrm{PE}$)-homeomorphism (but not $\mathrm{PL}$-homeomorphism).
 \item [(ii)] If $\pi(f) = 1$, $f$ is conjugate to a diffeomorphism through a $\mathrm{PL}$-homeomorphism.
\end{enumerate}
\end{thm}
\medskip

In particular, for $\textrm{PL}$-homeomorphism, we obtain:
\medskip

\begin{cor}\label{c:15} Let $f\in \mathrm{PL}(S^{1})$ with the ($D$)-property and irrational
rotation number $\alpha$. Assume that $\pi(f)=1$. Then
$f$ is conjugate to the rotation $R_{\alpha}$ through a $\mathrm{PL}$-homeomorphism.
\end{cor}
\bigskip

When $f$ does not satisfy the ($D$)-property, there is no rigidity; the conjugating homeomorphism is never piecewise $C^{1}$. 
\medskip

\begin{prop}\label{p:18}
Let $f\in \mathcal{P}(S^{1})$ with irrational 
rotation number. If $f$ does not satisfy the ($D$)-property, then it is not conjugate to a diffeomorphism through a piecewise 
$C^{1}$-homeomorphism of $S^{1}$. 
\end{prop}
\medskip

Actually, using a recent result du to Adouani \cite{A} and independently Dzhalilov et al. \cite{DMS},  one can say even more:

\begin{cor}\label{c:19}  Let $f\in \mathcal{P}(S^{1})$ with irrational 
rotation number. Assume that the derivatives $\textrm{Df}$ is absolutely
continuous on every continuity interval of $\textrm{Df}$. If $\pi_{s}(f)\neq 1$ then any homeomorphism map $h$ conjugating $f$ 
to a diffeomorphism of $S^{1}$ is a singular function i.e. it is continuous on $S^{1}$ and $\textrm{Dh}(x)=0$ a.e. 
with respect to the Lebesgue measure.
\end{cor}
\medskip

\textbf{Remark 1}. When $\pi_{s}(f) = 1$, the homeomorphism map $h$ conjugating $f$ to a diffeomorphism can be either 
a singular function or absolutely continued function.
Teplinsky gave in \cite{T} an example $f$ of 
$\textrm{PL}(S^{1})$ with four break points 
lying on pairwise distinct orbits and irrational rotation number of Roth number (but not of bounded type), that is 
conjugated to the rigid rotation by an absolutely continued function.
It is obvious that such example satisfies $\pi_{s}(f) = 1$ and does not satisfy the (D)-property.  
However, Herman has shown in \cite{H} (although not formulated as a statement) that a map $f\in \textrm{PL}(S^{1})$ 
with two breaks points lying on distinct orbits and irrational rotation number has
singular invariant measure; equivalently the homeomorphism $h$ conjugating $f$ to the rigid rotation is 
a singular function. 
 \medskip

This paper is organized as follows. 
Section 2 is devoted to the main technical part of the paper; we conjugate any class $P$-homeomorphism $f$ with several break points through a $\textrm{PQ}$-homeomorphism 
(resp. $\textrm{PE}$-homeomorphism) of $S^{1}$ to a class $P$-homeomorphism with prescribed break points on\textit{ pairwise 
distinct orbits}.
In Section 3, we study the case where $f$  
satisfies the ($D$)-property, we prove that it is conjugated through a 
$\textrm{PQ}$ (resp. $\textrm{PE}$)-homeomorphism of $S^{1}$ to a diffeomorphism. In particular, we study
the case where $f$ has two successive break points. Section 4 is devoted to class $P$-homeomorphism without the (D)-property.
\bigskip

%

%

\section{\bf Reduction to a class $P$-homeomorphisms with prescribed points on pairwise distinct orbits}
The aim of this section is to prove the following
\medskip

\begin{thm}\label{t:623}
Let $f\in \mathcal{P}(S^{1})$ with irrational rotation number, and let $\big(k_{0},\dots,k_{p}\big)\in \mathbb{Z}^{p+1}$. 
Then there exists a a $PQ$-homeomorphism (resp. $PE$-homeomorphism) $h\in \mathcal{P}(S^{1})$ such that $F:= h \circ f \circ h^{-1}\in \mathcal{P}(S^{1})$ with 
\
\\
- $C(F) \subset
\{h(f^{k_{i}}(c_{i}))= F^{k_{i}}(h(c_{i})); i=0,1, \dots, p\}$ \\
- $\sigma_{F}(F^{k_{i}}(h(c_{i}))) = \pi_{s,O_{f}(c_{i})}(f),~i=0,1, \dots, p$.
\end{thm}
\medskip

We need the following lemma, for completeness we present its proof.
\medskip

\begin{lem}
\label{l:22} Let $\sigma_{0},\dots, \sigma_{n}\in \mathbb{R}_{+}^{*}$ such that 
$\sigma_{0}\times\dots\times\sigma_{n} = 1$ 
and let $b_{0},\dots, b_{n}\in S^{1}$. Then there exists $L\in\mathrm{PL}(S^{1})$ 
with break points $b_{0},\dots, b_{n}$ and slopes
$\sigma_{L}(b_{0}) = \sigma_{0},\dots, \sigma_{L}(b_{n}) = \sigma_{n}$. In particular, $\pi_{s}(L)=1$.
\end{lem}
\medskip

\begin{proof} We let $b_{0} = p(\widetilde{b_{0}}),\dots, b_{n}= p(\widetilde{b_{n}})$, where $\widetilde{b_{0}} < \widetilde{b_{1}} < \dots < \widetilde{b_{n}} < \widetilde{b_{n+1}}$ be real numbers with
$\widetilde{b_{n+1}} = \widetilde{b_{0}}+1$, so $b_{n+1}=b_{0}$.
\medskip

 Define a $\textrm{PL}$-homeomorphism $\widetilde{L}$ on $[\widetilde{b_{0}}, \widetilde{b_{n+1}}]$ as follows:
\medskip

- $\widetilde{b_{0}},\dots, \widetilde{b_{n}}$ are the break points of $\widetilde{L}$. 

- $\sigma_{j}:= \sigma_{\widetilde{L}}(\widetilde{b_{j}})$ the jump of $\widetilde{L}$ in $\widetilde{b_{j}}$, $j=0,\dots, n$.
\medskip

Denote by

- $\lambda_{j} = D\widetilde{L}_{-}(\widetilde{b_{j}})$ the slope of $\widetilde{L}$ on
$[\widetilde{b_{j-1}}, \widetilde{b_{j}}[$, $j=1,\dots, n$

- $\lambda_{0}$ the slope of $\widetilde{L}$ on
$[\widetilde{b_{n}}, \widetilde{b_{n+1}}[$, \ $\widetilde{L}(x) = \lambda_{0}(x-\widetilde{b_{n+1}})+\widetilde{b_{n+1}}, \; \ x\in [\widetilde{b_{n}}, 
\widetilde{b_{n+1}}[$. 

-  $\widetilde{L}(\widetilde{b_{0}})=\widetilde{b_{0}}$.

One has $\sigma_{j}= \frac{\lambda_{j}}{\lambda_{j+1}}$ and
$\frac{\lambda_{0}}{\lambda_{j}}=
\frac{\lambda_{0}}{\lambda_{1}}\times\dots \times\frac{\lambda_{j-1}}{\lambda_{j}} = \sigma_{0}\times\dots \times\sigma_{j-1}$. Hence 
$$\lambda_{j}= (\sigma_{0}\times\dots \times\sigma_{j-1})^{-1}\lambda_{0}, \ j=1,\dots, n.$$ 
To determine $\lambda_{0}$, we have the identity 
$$\widetilde{L}(\widetilde{b_{n}}) = \widetilde{L}(\widetilde{b_{0}}) + 
\underset{j=0}{\overset{n-1}{\sum}}\lambda_{j+1}(\widetilde{b_{j+1}}-\widetilde{b_{j}})=$$ 
$$\widetilde{b_{0}} + 
\lambda_{0}\underset{j=0}{\overset{n-1}{\sum}}(\sigma_{0}\times\dots \times\sigma_{j})^{-1}(\widetilde{b_{j+1}}-
\widetilde{b_{j}})=$$ $$-\lambda_{0}(\widetilde{b_{n+1}}-\widetilde{b_{n}}) + \widetilde{b_{n+1}}.$$  
Thus 

$$\lambda_{0}\left(\underset{j=0}{\overset{n-1}{\sum}}(\sigma_{0}\times\dots \times\sigma_{j})^{-1}(\widetilde{b_{j+1}}-
\widetilde{b_{j}}) + 
(\widetilde{b_{n+1}}-\widetilde{b_{n}})\right) = 1.$$

Hence 

$$\lambda_{0} = \dfrac{1}{\left(\underset{j=0}{\overset{n-1}{\sum}}(\sigma_{0}\times\dots \times\sigma_{j})^{-1}
(\widetilde{b_{j+1}}-\widetilde{b_{j}}) + 
(\widetilde{b_{n+1}}-\widetilde{b_{n}})\right)}.$$
\bigskip

Then $\widetilde{L}$ is a homeomorphism of $[\widetilde{b_{0}}, \widetilde{b_{n+1}}]$. 
The $PL$-homeomorphism $L$ of $S^{1}$ is then defined by its 
lift $\widetilde{L}$ restricted to $[\widetilde{b_{0}}, \widetilde{b_{n+1}}]$. 
 \end{proof}
 \medskip
 
\begin{proof}[Proof of Theorem \ref{t:623}] Set for $i=0, \dots, p$ and  $k \in \mathbb{Z}$:

$$ m_{i} = \min(0,k_{i}),~n_{i}=\max (k_{i},N_{i})$$

$$ \sigma(f) = \prod_{i=0}^{p}\prod_{k \in \mathbb{Z}}\sigma_{k,i}(f),$$

where $$\sigma_{k,i}(f) = \left\{
               \begin{array}{ll}
                 \prod_{j \geq k}a_{j,i}(f), & \textrm{if}~~ k>k_{i}\\\\

                \dfrac{ 1}{\prod_{j < k}a_{j,i}(f)}, & \textrm{if} ~~k \leq k_{i}
               \end{array}
             \right.$$
\
\\
and $$a_{k,i}(f)=\sigma_{f}(f^{k}(c_{i})).$$
\
\\
 Then we obtain $$
 \sigma(f)= \prod_{i=0}^{p}(\pi_{s,O_{f}(c_{i})}(f))^{-k_{i}} \prod_{j \in
\mathbb{Z}}(a_{j,i}(f))^{j}$$
 
 Indeed, we have
 
 $$a_{k,i}(f)=1, ~~\textrm{if} ~~k<0 ~~\textrm{or}~~ k>N_{i}$$

$$\sigma_{k,i}(f)=1,~~ \textrm{if} ~~k<m_{i} ~~\textrm{or} ~~k>n_{i}$$

\begin{eqnarray*}
  \prod_{k \leq k_{i}}\sigma_{k,i}(f) &=&   \prod_{k \leq k_{i}}\left(\prod_{j < k}(a_{j,i}(f))^{-1}\right) \\
    & = &  \prod_{j < k_{i}}\left(\prod_{j < k \leq k_{i}}(a_{j,i}(f))^{-1}\right)\\
& = &   \prod_{j < k_{i}}(a_{j,i}(f))^{j-k_{i}}\\
& = &   \prod_{p < 0}\left(a_{p+k_{i},i}(f)\right)^{p}
\end{eqnarray*}
Similarly, $$\prod_{k >k_{i}}\sigma_{k,i}(f) = \prod_{p \geq
0}\left(a_{p+k_{i},i}(f)\right)^{p}$$
 So
\begin{eqnarray*}
 \prod_{k \in \mathbb{Z} }\sigma_{k,i}(f)&=&\prod_{p \in
\mathbb{Z}}(a_{p+k_{i},i}(f))^{p} \\
& = & \prod_{j \in \mathbb{Z}}(a_{j,i}(f))^{j-k_{i}} \\
&= & (\pi_{s,O_{f}(c_{i})}(f))^{-k_{i}} \prod_{j\in
\mathbb{Z}}(a_{j,i}(f))^{j}
\end{eqnarray*}
Therefore
\begin{eqnarray*}
 \sigma(f) &=&\prod_{i=0}^{p}\prod_{k \in \mathbb{Z} }\sigma_{k,i}(f) \\
&= & \prod_{i=0}^{p}(\pi_{s,O_{f}(c_{i})}(f))^{-k_{i}} \prod_{j \in
\mathbb{Z}}(a_{j,i}(f))^{j}
\end{eqnarray*}
\medskip

Now, set $$b_{k,i}(f)=\dfrac{\sigma_{k+1,i}(f)}{\sigma_{k,i}(f)}a_{k,i}(f).$$ Then we obtain
$$b_{k,i}(f) =\left\{
  \begin{array}{ll}
  \pi_{s,O_{f}(c_{i})}(f), & \textrm{if} ~~ k=k_{i}\\\\

    1, & \textrm{otherwise}
  \end{array}
\right. $$
\medskip

 Indeed:
\\
\\
For $k>k_{i}$,
\begin{eqnarray*}
 \sigma_{k,i}(f) &=&\prod_{j \geq k}a_{j,i}(f) \\
& = & a_{k,i}(f)\prod_{j \geq k+1 }a_{j,i}(f) \\
& = & a_{k,i}(f)~\sigma_{k+1,i}(f)
\end{eqnarray*}
For $k<k_{i}$,
\begin{eqnarray*}
 \sigma_{k,i}(f) &=& \dfrac{1}{\prod_{j < k}a_{j,i}(f)} \\\\
& = & \dfrac{a_{k,i}(f)}{\prod_{j < k+1 }a_{j,i}(f)} \\\\
& = & a_{k,i}(f)~\sigma_{k+1,i}(f)
\end{eqnarray*}
For $k=k_{i}$,

\begin{eqnarray*}
b_{k,i}(f)&= &\dfrac{\sigma_{k+1,i}(f)}{\sigma_{k,i}(f)} a_{k,i}(f) \\\\
& = & \dfrac{\prod_{j \geq k +1}a_{j,i}(f)}{ \big (\prod_{j < k } a_{j,i}(f)\big)^{-1}} a_{k,i}(f) \\\\
& = & ~~\prod_{j\in \mathbb{Z}} a_{j,i}(f)\\\\
& =  & \pi_{s,O_{f}(c_{i})}(f)
\end{eqnarray*}
\medskip

We distinguish two cases.
\medskip

\
\\
\textbf{Case 1}: $\sigma(f)=1$. By Lemma \ref{l:22}, there exists $L\in
\textrm{PL}(S^{1})$ with the following properties:
\medskip
\
\\
(i) $L(0)=0$ \\
(ii) $C(L)\subset \{f^{k}(c_{i}):~m_{i} \leq k \leq n_{i},~0 \leq i
\leq p \}$ \\
(iii) $\sigma_{L}(f^{k}(c_{i}))=\sigma_{k,i}(f)$
\medskip
\
\\
We let $F = L\circ f\circ L^{-1}$. A priori, the break points of $F$ are: \\
- The break points of $L^{-1}$: $L(f^{k}(c_{i})), ~m_{i} \leq k
\leq n_{i},~0 \leq i \leq p$, \\
- The image by $L$ of break points of $f$: $L(f^{k}(c_{i})),~m_{i}-1
\leq k \leq n_{i},~0 \leq i \leq p$. \\
Therefore the possible break points of $F$ are among:
$L(f^{k}(c_{i})), ~m_{i} \leq k \leq n_{i},~0 \leq i \leq p$. Compute the jumps of $F$ in these points:
\begin{align*}
\sigma_{F}(L(f^{k}(c_{i}))) & = & \ \dfrac{\sigma_{L}(f(f^{k}(c_{i})))~\sigma_{f}(f^{k}(c_{i}))}{\sigma_{L}(f^{k}(c_{i}))} \\
&=  & \ \dfrac{\sigma_{k+1,i}(f)~a_{k,i}(f)}{\sigma_{k,i}(f)}\\
&= & \  b_{k,i}(f) \\
& =& \begin{cases}
\pi_{s,O_{f}(c_{i})}(f), & \textrm{if} ~~k=k_{i}\\
                  1, & \ \textrm{otherwise}
\end{cases}
\end{align*}
We conclude that $C(F)\subset \{L(f^{k}(c_{i})):~0 \leq i \leq p \}$
with $\sigma_{F}(L(f^{k}(c_{i})))=\pi_{s,O_{f}(c_{i})}(f)$. 
\medskip

\textbf{Case 2}: $\sigma(f)\neq 1$. Set $\sigma =\sigma(f)$ and define $u = R_{c}\circ g_{\sigma}\circ
R_{c}^{-1} $ (resp. $u= R_{c}\circ h_{\sigma}\circ
R_{c}^{-1} $), where $c=f^{N_{0}+1}(c_{0})$. Then $u$ is a particular $PQ$-homeomorphism (resp. $PE$-homeomorphism) 
with one break point $c$ and such that:
$\sigma_{u}(c)=\sigma$. We let
$F = u\circ f\circ u^{-1}$. A priori, the break points of $F$ are:\\

- The break point of $u^{-1}: ~u(f^{N_{0}+1}(c_{0}))$ \\
- The image by $u$ of break points of $f$:
$u(f^{k}(c_{i})), ~0 \leq k \leq N_{i}, ~0\leq i \leq p$ \\
- The image by $u\circ f^{-1}$ of the break point of $u
:~u(f^{N_{0}}(c_{0}))$ \\
Therefore the possible break points of $F$ are among \\
$u(f^{k}(c_{i})),~0 \leq k \leq N_{i},~1 \leq i \leq p, \ \textrm{
and } u(f^{k}(c_{0})), ~0 \leq k \leq N_{0}+1$.\\ 

Compute the jumps of $F$ in these points:\\
For $~0 \leq k \leq N_{i},~1 \leq i \leq p$,
\begin{align*}
\sigma_{F}(u(f^{k}(c_{i}))) & =
\dfrac{\sigma_{L}(f(f^{k}(c_{i})))~\sigma_{f}(f^{k}(c_{i}))}{\sigma_{u}(f^{k}(c_{i}))}\\\\
& =   a_{k,i}(f)\\\\
\sigma_{F}(u(f^{N_{0}}(c_{0}))) & =
\dfrac{\sigma_{u}(f(f^{N_{0}+1}(c_{0})))~\sigma_{f}(f^{N_{0}}(c_{0}))}{\sigma_{u}(f^{N_{0}}(c_{0}))}\\\\
& =  \dfrac{\sigma(f)~a_{N_{0},0}(f)}{1}\\\\
& =  \sigma(f)~a_{N_{0},0}(f)\\\\
 \sigma_{F}(u(f^{N_{0}+1}(c_{0})))& =
\dfrac{\sigma_{u}(f(f^{N_{0}+2}(c_{0})))~\sigma_{f}(f^{N_{0}+1}(c_{0}))}{\sigma_{u}(f^{N_{0}+1}(c_{0}))}\\\\
& =  \dfrac{1 \times 1}{\sigma(f)}\\\\
& =  \dfrac{1}{\sigma(f)}\\\\
\sigma_{F}(u(f^{k}(c_{0}))) &  =  a_{k,0}(f),~0 \leq k < N_{0}
\end{align*}
Let $a_{k,i}(F): = \sigma_{F}(F^{k}(d_{i})) =
\sigma_{F}(u(f^{k}(c_{i})))$, where $~d_{i} = u(c_{i}),~\textrm{for} ~0 \leq k \leq
N_{0},~ 0 \leq i \leq p$. Then, \\\\
 $$\sigma_{k,i}(F) = \left\{
   \begin{array}{ll}
\prod_{j \geq k}a_{j,i}(F), &  \textrm{if} ~k>k_{i} \\
 \dfrac{1}{\prod_{j < k}a_{j,i}(F)}, & \textrm{if} ~k \leq k_{i}
  \end{array}
\right.$$ For $1\leq i \leq p$,

\begin{eqnarray*}
  \prod_{k \in \mathbb{Z}} \sigma_{k,i}(F)& = & (\pi_{s,O_{F}(d_{i})}(F))^{-k_{i}} ~\prod_{j \in \mathbb{Z}}~
(a_{j,i}(F))^{j}\\
  &=& (\pi_{s,O_{f}(c_{i})}(f))^{-k_{i}} ~\prod_{j \in \mathbb{Z}}~
(a_{j,i}(f))^{j} \\
& = & \prod_{k \in \mathbb{Z}} \sigma_{k,i}(f)
\end{eqnarray*}

\begin{eqnarray*}
 \prod_{k \in \mathbb{Z}} \sigma_{k,0}(F)& = &
(\pi_{s,O_{F}(d_{0})}(F))^{-k_{0}} ~\prod_{j \in \mathbb{Z}}~
(a_{j,0}(F))^{j}\\\\
& = &
(\pi_{s,O_{f}(c_{0})}(f))^{-k_{0}}~(a_{N_{0},0}(F))^{N_{0}}~(a_{N_{0}+1,0}(F))^{N_{0}+1}
~\prod_{ j \in \mathbb{Z},~ j \neq N_{0},~j \neq
N_{0}+1}~(a_{j,0}(F))^{j}\\\\
& = &
(\pi_{s,O_{f}(c_{0})}(f))^{-k_{0}}~\left(\sigma(f)~a_{N_{0},0}(f)\right)^{N_{0}}~
\left(\dfrac{a_{N_{0}+1,0}(f)}{\sigma(f)}\right)^{N_{0}+1}
~\prod_{j\in \mathbb{Z},~ j\neq N_{0},~j \neq
N_{0}+1}~(a_{j,0}(f))^{j}\\\\
& = & \dfrac{1}{\sigma(f)}~(\pi_{s,O_{f}(c_{0})}(f))^{-k_{0}}~\prod_{j\in \mathbb{Z}}~(a_{j,0}(f))^{j}\\\\
& = & \dfrac{1}{\sigma(f)}~\prod_{j\in \mathbb{Z}}~\sigma_{j,0}(f) 
\end{eqnarray*}

Therefore \\
\begin{eqnarray*} \sigma(F) & = & \prod_{i=0}^{p}~\prod_{j\in
\mathbb{Z}}~\sigma_{j,i}(F)\\
& = & \dfrac{1}{\sigma(f)}~\prod_{i=0}^{p}~\prod_{j\in
\mathbb{Z}}~\sigma_{j,i}(f)\\
& = & 1
\end{eqnarray*}

We conclude that $F \in \mathcal{P}(S^{1})$ that satisfies $\sigma(F) = 1$ and with maximal $F$-connections $M_{0}(F)
= [u(c_{0}), \dots, F^{N_{0}+1}(u(c_{0}))]$ and $M_{i}(F) =
[u(c_{i}), \dots, F^{N_{i}}(u(c_{i}))]$, for $1 \leq i \leq p$.
 Then, by the case 1, there exists $L \in
PL(S^{1})$ that conjugates $F$ to a class $P$-homeomorphism $G
= L\circ F \circ L^{-1}$ with $C(G)\subset \{G^{k_{i}}[L(u(c_{i}))]:
~0 \leq i \leq p\}$ and $\sigma_{G}(G^{k_{i}}([L(u(c_{i}))]) =
\pi_{s,O_{f}(c_{i})}(f),~0 \leq i \leq p$. Moreover $h: =
L\circ u$ is a $PQ$-homeomorphism (resp. $PE$-homeomorphism) that
conjugates $f$ to $G$ with $C(G)\subset\{G^{k_{i}}(h(c_{i})):
~0 \leq i \leq p\}$ and $\sigma_{G}(G^{k_{i}}(h(c_{i}))) =
\pi_{s,O_{f}(c_{i})}(f),~0 \leq i \leq p$. This completes the proof.
\end{proof}
\medskip

\begin{cor}
\label{c:21} Let $f\in \mathcal{P}(S^{1})$ with irrational rotation number. Then, there exists 
$h\in \mathcal{P}(S^{1})$ such that:
 $F = h\circ f \circ h^{-1}\in \mathcal{P}(S^{1})$ with $C(F)\subset\{h(c_{0}),\dots,h(c_{p})\}$, where 
 $c_{0},\dots, c_{p}\in C(f)$ are on pairwise distinct orbits.
 Moreover $\sigma_{F}(h(c_{i})) = \pi_{s,O_{f}(c_{i})}(f),~i=0,1, \dots, p$.
\end{cor}
\smallskip

\begin{proof} Take $k_{i}= 0$ for all $i$ in Theorem \ref{t:623}. So we get
 $F = h\circ f \circ h^{-1}\in \mathcal{P}(S^{1})$
with $C(F)\subset\{h(c_{0}),\dots, h(c_{p})\}$, where  $c_{0},\dots, c_{p}\in C(f)$
are on pairwise distinct orbits.
\end{proof}
\medskip

\section{\bf Class $P$-homeomorphisms with the (D)-property} 
\medskip

\subsection{Proof of Theorem \ref{t:23}}
\medskip

\begin{lem}\label{l31} Let $f\in \mathcal{P}(S^{1})$ with irrational rotation number. If $f$ has the (D)-property then 
$\sigma(f)= \pi(f)$.
 \end{lem}
\smallskip
 
 \begin{proof} We have $\sigma(f)= \prod_{i=0}^{p}(\pi_{s,O_{f}(c_{i})}(f))^{-k_{i}} \prod_{j\in
\mathbb{Z}}(a_{j,i}(f))^{j}$. Since $\pi_{s,O_{f}(c_{i})}=1$ and $\prod_{j\in
\mathbb{Z}}(a_{j,i}(f))^{j}= \pi(f)$, for every $i=0,\dots, p$, thus $\sigma(f)= \pi(f)$.
 \end{proof}
\smallskip
\
\\
{\it Proof of Theorem \ref{t:23}}. From the Corollary \ref{c:21}, it follows that $F: h\circ f\circ h^{-1}$ is a 
diffeomorphism since  
$\sigma_{F}(h(c_{i})) = \pi_{s,O_{f}(c_{i})}(f)=1,~i=0,1, \dots, p$.
Now by the proof of Theorem \ref{t:623}, $h$ is a $\textrm{PL}$-homeomorphism if $\sigma(f) = 1$ and a 
$\textrm{PQ}$ (resp. $\textrm{PE}$)-homeomorphism if $\sigma(f)\neq 1$. We conclude by the Lemma \ref{l31}. \qed  
\medskip

\textbf{Remark 2.} The $\textrm{PE}$ (resp. $\textrm{PQ}$)-homeomorphism $h = L\circ u$ 
that conjugates $f$ to a diffeomorphism can be chosen so that its rotation number is $0$. Indeed, 
let $u = R_c\circ g_{\sigma}\circ R_c^{-1}$ (resp. $R_c\circ h_{\sigma}\circ R_c^{-1}$), for some  
$\sigma\in \mathbb{R}_{+}^{*}\setminus\{1\}$ and $c\in S^{1}$. Set $d = R_{c}(0)$ and choose $L\in \mathrm{PL}(S^{1})$ such that 
$L(d)= d$. Then $h(d)=d$ and so $h$ has a rotation number $0$. 
\medskip

\subsection{\bf Case of two break points} 
\medskip

Let $f\in \mathcal{P}(S^{1})$ with irrational
rotation number $\alpha$ and with two break points $b$ and  $f(b)$. Assume that $f$ satisfies the ($D$)-property. 
We give a direct conjugation $h$ from $f$ to a diffeomorphism. This conjugation $h$ is different from that constructed in 
the proof of Theorem \ref{t:623}.
\medskip
We let $b^{\prime} = f(b)$ and $\sigma = \pi(f)^{-1}$. Define 
$h: = R_{b^{\prime}}\circ h_{\sigma}^{-1}\circ (R_{b^{\prime}})^{-1}$. 
 Then  $h$ is a $PE$-homeomorphism with one break point 
 $b^{\prime}$ such that: $\sigma_{h}(b^{\prime}) = \ \sigma^{-1}$. 
 \medskip
 
\begin{prop}\label{p41} Let $f\in \mathcal{P}(S^{1})$ with two break points $b$ and  $f(b)$ and irrational rotation number. Assume that 
$ \pi(f)\neq 1$. Then $F: = h\circ f\circ h^{-1}\in \mathcal{P}(S^{1})$ with $C(F)\subset \{h(b)\}$ such that $\sigma_{F}(h(b)) = 
\pi_{s,O_{f}(b)}(f)$.
\
\\
In particular if $f$ satisfies the ($D$)-property then $F$ is a diffeomorphism.
 \end{prop} 
\medskip

\begin{proof} We let $F = h\circ f\circ h^{-1}$. Then
 $$\sigma_{F}(h(b^{\prime})) = \
\frac{\sigma_{h}(f(b^{\prime}))
\sigma_{f}(b^{\prime})}{\sigma_{h}(b^{\prime})}.$$

As $\sigma_{h}(f(b^{\prime})) = 1$ and $\sigma_{f}(b^{\prime}) = \sigma^{-1}$, so  
$\sigma_{F}(h(b^{\prime})) = 1$.

On the other hand, we have:
$$\sigma_{F}(h(b)) = \
\frac{\sigma_{h}(b^{\prime})
\sigma_{f}(b)}{\sigma_{h}(b)}.$$

As $\sigma_{h}(b^{\prime}) = \sigma^{-1}$, $\sigma_{h}(b) =1$ and 
$\sigma_{f}(b) = \frac{\pi_{s,O_{f}(b)}(f)}{\sigma^{-1}}$

then $\sigma_{F}(h(b)) = \pi_{s,O_{f}(b)}(f)$  and $C(F)\subset \{h(b)\}$.

 In particular, if $f$ satisfies the ($D$)-property then $\pi_{s,O_{f}(b)}(f)=1$. So $F$ has no break points and $F$ is a diffeomorphism.
 \end{proof}
\bigskip

\begin{cor}\label{l:43} Let  $f\in \mathrm{PL}(S^{1})$ with irrational
rotation number $\alpha$ and with two break points $b$ and  $f(b)$. Then $\pi(f)\neq 1$ and
$h\circ f\circ h^{-1}$ is the rotation $R_{\alpha}$.
\end{cor}
\smallskip

\begin{proof}
One has $\pi(f)= \sigma_{f}(b^{\prime})\neq 1$. Moreover $f$ satisfies the ($D$)-property.
One can check that
$h_{\sigma}^{-1}\circ \left( (R_{b^{\prime}})^{-1}\circ f\circ R_{b^{\prime}}\right)\circ h_{\sigma} = R_{\alpha}$ and therefore $h\circ f\circ h^{-1} = R_{\alpha}$.
\end{proof}

\section{\bf Class $P$-homeomorphisms without the (D)-property} 
\bigskip

{\it Proof of Proposition \ref{p:18}}. Suppose that there is a piecewise $C^{1}$-homeomorphism $h$ that conjugates $f$ to a diffeomorphism $F$:
$f= h^{-1}\circ F\circ h$. Since the rotation number is irrational, $h^{-1}$ is also piecewise $C^{1}$. 
As $h$ and $h^{-1}$ have the same number $p$ of break points and $f^{n}= h^{-1}\circ F^{n}\circ h$ then $f^{n}$ has at most
$2p$ break points for every $n\in \mathbb{Z}$. So by (\cite{AM2}, Proposition 2.5), $f$ satisfies
the (D)-property, a contradiction. \qed
\bigskip

{\it Proof of Corollary \ref{c:19}}. This follows directly from (\cite{A}, Main Theorem) since $\pi_{s}(F)=1$ for any
diffeomorphism $F$ of $S^{1}$. \qed
\bigskip

\bibliographystyle{amsplain}

\begin{thebibliography}{9}
\bibitem{AM2} A. Adouani, H. Marzougui, \emph{On piecewise smoothness of conjugacy of class P circle
homeomorphisms to diffeomorphisms and rotations}, Dyn. Syst., \textbf{27} (2012), 169--186.
\bibitem{A} A. Adouani, \emph{Conjugation between circle maps with several break points},  Ergod. Th. and Dynam. Sys., to appear (2015).
\bibitem{aD32} A. Denjoy,  \emph{Sur les courbes d\'efinies par les \'equations diff\'erentielles \`a la surface du tore}, J. Math. Pures Appl.,
\textbf{11} (1932), 333--375.
\bibitem{DMS} A. Dzhalilov, D. Mayer, U. Safarov, \emph{On the conjugation of piecewise smooth circle homeomorphisms with a finite
number of break points}, Nonlinearity, \textbf{28} (2015), doi:10.1088/0951-7715/28/7/2441

\bibitem{H} M. Herman, \emph{Sur la conjugaison diff\'erentiable
des diff\'eomorphismes du cercle \`a des rotations}, Pub. Math. Inst. Hautes Etudes Sci., \textbf{49} (1979), 5--234.
\bibitem{T} A. Teplinsky, \emph{A circle diffeomorphism with breaks that is smoothly
linearizable}, arXiv: 1506.06617v1 (2015).
\end{thebibliography}

\end{document}